\documentclass[12pt,a4paper]{amsart}
\usepackage{amssymb}
\usepackage{verbatim}
\usepackage{esint} 
\usepackage{graphicx}
\newtheorem{theorem}{Theorem}[section]
\newtheorem{lemma}[theorem]{Lemma}

\newtheorem{corollary}[theorem]{Corollary}
\newtheorem{fact}[theorem]{Fact}

\theoremstyle{definition}

\newtheorem{conjecture}[theorem]{Conjecture}

\theoremstyle{remark}
\newtheorem{remark}[theorem]{Remark}

\numberwithin{equation}{section}

\renewcommand{\dim}{\mathop{\mathrm{dim}}}

\newcommand{\dist}{\mathrm{dist}}

\newcommand{\R}{\mathbb{R}}
\newcommand{\N}{\mathbb{N}}

\newcommand{\X}{\mathrm{X}}

\newcommand{\E}{\mathbb{E}}
\renewcommand{\P}{\mathbb{P}}

\renewcommand{\ker}{\mathrm{Ker}}

\begin{document}

\title[Quasihyperbolic Geodesics]{On Quasihyperbolic Geodesics in Banach Spaces}

\begin{abstract}
We study properties of quasihyperbolic geodesics on Banach spaces. For example, we show that in a strictly convex Banach space with the Radon-Nikodym property, the quasihyperbolic geodesics are unique. We also give an example of a convex domain $\Omega$ in a Banach space such that there is no geodesic between any given pair of points $x, y  \in \Omega\,.$ In addition, we prove that if $\X$ is a uniformly convex Banach space and its modulus of convexity is of a power type, then every geodesic of the quasihyperbolic metric, defined on a proper subdomain of $\X$, is smooth.
\end{abstract}

\author[Antti Rasila]{Antti Rasila$^1$}
\thanks{$^1$ Corresponding author.}
\address{Antti Rasila,  Department of Mathematics, Hunan Normal University, Changsha, Hunan 410081, People's Republic of China, and Aalto University, Institute of Mathematics, P.O. Box 11100, FI-00076 Aalto, Finland} 
\email{antti.rasila@iki.fi}

\author[Jarno Talponen]{Jarno Talponen}
\address{Jarno Talponen, University of Eastern Finland, Department of Physics and Mathematics, P.O. Box 111 FI-80101, Joensuu, Finland} 
\email{talponen@iki.fi}
\date{\today}
\maketitle

\section{Introduction}

The \emph{quasihyperbolic metric} in $\R^n$ is a generalization of the hyperbolic
metric. This metric was first studied in $\R^n$ by Gehring  in joint publications \cite{GehringOsgood79,GehringPalka76} with his students Palka and Osgood in late 1970's.  Since its discovery, the quasihyperbolic metric has been widely applied in the study of geometric function theory \cite{Heinonen01,Vuorinen88}. At the same time,
many basic questions of the quasihyperbolic geometry have remained open. 
Only very recently several authors have studied questions such as convexity of balls of
small radii, uniqueness of geodesics, and quasihyperbolic trigonometry of plane domains.
See for instance \cite{ChenChenQian12, HuangPonnusamyWangWang10, Klen08,Klen09,  MartioVaisala11, RasilaTalponen12, Vaisala05,  Vuorinen10}.  We refer to \cite{Vuorinen88} for the basic properties of this metric in $\R^n.$

In Banach spaces, the properties of the quasihyperbolic metric  were first studied by  V\"ais\"al\"a in a series of articles in 1990's \cite{Vaisala90,Vaisala91,Vaisala92,Vaisala98,Vaisala99}. The quasihyperbolic metric is a crucial tool for studying quasiconformal mappings in the in finite dimensional Banach spaces because quasiconformality is defined in terms of it. Moreover, the basic tools of the finite dimensional theory in $\R^n$, such as the conformal modulus and finite dimensional measure theory, are not available in the case of infinite dimensional Banach spaces. See \cite{KRT} for 
discussion and motivations on the topic.

A result of Gehring and Osgood shows that in domains of $\R^n$ there is always a quasihyperbolic geodesic between any two points \cite{GehringOsgood79}. It was proved by Martin in \cite{Martin85}, that the geodesics of the quasihyperbolic metric in $\R^n$ are smooth. The proof makes use of M\"obius transformations, and it does not work in the Banach space setting. In this paper, our aim is to generalize these results to the Banach space  setting. 

A result of Martio and V\"ais\"al\"a \cite{MartioVaisala11} shows that if $\Omega$ is a convex domain of a uniformly convex Banach space, then each pair of points $x,y\in \Omega$ is joined by a unique quasihyperbolic geodesic. The existence argument is based on the fact that each uniformly convex Banach space is reflexive 
and the unit ball of reflexive space is compact in the weak topology. The geodesic is obtained from 
a sequence of short paths.

The first of our main results, Theorem \ref{main1}, shows by following the arguments of our earlier paper \cite{RasilaTalponen12} that in a convex domain $\Omega$ of a reflexive strictly convex Banach space the quasihyperbolic geodesics are unique. The argument boils down to taking averages of paths. We also prove by a counterexample 
that there does not necessarily exist any quasihyperbolic geodesics between any two points in a convex domain of a non-reflexive Banach space. This appears a somewhat surprising fact and it settles a problem posed in \cite{Vaisala05} in the negative.

Finally, we show that quasihyperbolic geodesics in a uniformly convex Banach space, with a power type modulus of convexity, are $C^1$ smooth. This result generalizes a classical theorem of Martin \cite{Martin85} in the Euclidean setting, as well as V\"ais\"al\"a's recent work in Hilbert spaces \cite{Vaisala07}.

\subsection{Preliminaries}

Here we consider Banach spaces $\X$ over the real field. We refer to \cite{Diestel}, \cite{DiestelUhl77}, \cite{Banach_space_book}, \cite{LindTza} and \cite{Vuorinen88} for suitable background information. 

Suppose that $\X$ is a Banach space with $\dim \X \ge 2$, and let $\Omega\subsetneq \X$ be an open path-connected domain. We call a continuous function $w\colon \Omega\to (0,\infty)$ a \emph{weight function}. Then the \emph{$w$-length} of a rectifiable arcs $\gamma \subset \Omega$ is defined by
\[
\ell_w(\gamma)=\int_{\gamma}w\big(\gamma(t)\big)\ dt.
\]
We also define a (conformal) metric $d_{w}$ on $\Omega$ by
\[
d_{w}(x,y) = \inf_{\gamma} \ell_w(\gamma)
\]
where the infimum is taken over rectifiable arcs $\gamma$ joining $x$ and $y$ in $\Omega$. It is clear that $d_w$ defines a metric in $\Omega$. If the infimum is attained for an arc $\gamma$, then we call $\gamma$ a \emph{$d_w$-geodesic}. It is easy to see that $d_w$-geodesics do not always exist.

For the weight function
\[
w(x)= \frac{1}{\dist(x,\partial\Omega)},
\]
we denote $d_w = k_{\Omega}$, and call $k_\Omega$ the quasihyperbolic metric of $\Omega$.
The open quasihyperbolic balls are given by 
\[D(x,r)=\{y\in \Omega\colon k(x,y)<r\}.\]

A set $\Omega\subset \X$ is called \emph{convex} if the line segment
\[
[x,y] := \{ tx + (1-t)y : t\in [0,1]\} \subset \Omega\textrm{ for all }x,y\in\Omega.
\] 
Note that the use of notation $[x,y]$ here is different from some texts dealing with Banach spaces.
The norm of a Banach space is said to be \emph{strictly convex} if the unit sphere contains 
no proper line segment, i.e. $\|x+y\|<2$ for $\|x\|=\|y\|=1,\ x\neq y$.

The \emph{modulus of convexity} $\delta_{\X}(\epsilon),\ 0<\epsilon \leq 2,$ is defined by
\[
\delta_{\X}(\epsilon)=\inf\big\{1-\|x+y\|/2:\ x,y\in \X,\ \|x\|=\|y\|=1,\ \|x-y\|=\epsilon\big\}.
\]
A Banach space $\X$ is called \emph{uniformly convex} if $\delta_{\X}(\epsilon)>0$ for $\epsilon>0$ 
(see Figure \ref{fig1}).

\begin{figure}

\includegraphics[width=8cm]{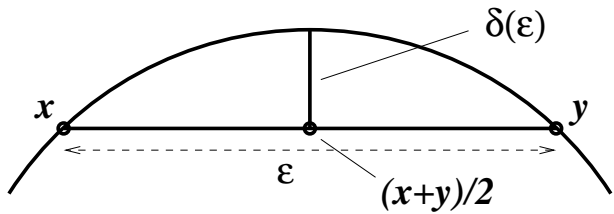}

\caption{The modulus of convexity controls the norm of averages of vectors.}\label{fig1}
\end{figure}

A space $\X$ is uniformly convex of power type $p\in [2,\infty)$ if $\delta_{\X}(\epsilon)\geq K\epsilon^{p}$ for some $K>0$. For example any reflexive $L^p$ space has this property.

A Banach space $\X$ is said to have the \emph{Radon-Nikodym property} (RNP) if the Radon-Nikodym Theorem holds for $\X$. That is, if each Banach-valued measure $\mu\colon \Sigma\to \X$, where $\Sigma$ is a sigma algebra and $|\mu|<\infty$, has the representation
\[
\mu(E)=\int_{E}\phi(x)\ d|\mu|(x),
\]
for all $E\in \Sigma$ and some Bochner integrable $\phi\colon \Omega \to \X$.  All reflexive spaces and separable dual spaces have the RNP. For example, the space $c_{0}$ of sequences converging to $0$ with the $\sup$-norm fails the RNP. 

A path $\gamma\colon [0,1]\to \X$ is differentiable at $x_{0}\in (0,1)$ if
\[
\lim_{t\to 0}\frac{\gamma(x_{0}+t)-\gamma(x_{0})}{t}\quad \mathrm{exists\ in}\ \X.
\]
This derivative is denoted by $\gamma'$. A Banach space $\X$ has the RNP if and only if each absolutely continuous path $\gamma\colon [0,1]\to \X$ is differentiable a.e. In such a case the fundamental theorem of analysis holds:
\[\gamma(t)=\gamma(0)=\int_{0}^{t}\gamma'\ ds\]
where the integral is taken in the Bochner sense. We say that path $\gamma\colon [a,b]\to \X$ is $C^1$-smooth if 
the derivative $\gamma'$ exists and is continuous away from the endpoints $a$ and $b$ 
(where the derivative is not defined in the usual sense). A \emph{geodesic} is a path in a metric space which has 
the least path length among the paths connecting its endpoints. We call a metric space \emph{geodesic} if 
each pair of points can be joined by a geodesic (not necessarily unique).

\section{Existence and uniqueness of geodesics}

In general domains the quasihyperbolic geodesics need not be unique. A simple example is the space punctured at a point. For the uniqueness one needs additional assumptions such as convexity. Another related topic is the convexity of balls with small radii which was recently proved by Kl\'en in the special case of the punctured space \cite{Klen08,Klen09}, and by V\"ais\"al\"a \cite{Vaisala12} for general plane domains.

By following the arguments in \cite{RasilaTalponen12} one can check the following:

\begin{theorem}
\label{main1}
In a convex domain $\Omega$ of a strictly convex Banach space with the RNP the quasihyperbolic geodesics are unique.
\end{theorem}

\begin{remark}
The existence of geodesics in the reflexive case was given by V\"ais\"al\"a in \cite{Vaisala05}.
\end{remark}

\begin{proof}[Proof of Theorem \ref{main1}]

For a path $\gamma\colon [0,1]\to \Omega$ and $t_1,t_2\in[0,1]$ denote by $\ell_{\mathrm{qh}}(\gamma,t_1,t_2)$ the quasihyperbolic length of $\gamma([t_1,t_2])$. We write $\ell_{\mathrm{qh}}(\gamma)$ for $\ell_{\mathrm{qh}}(\gamma,0,1)$. 

Suppose that $\gamma_1,\gamma_2\colon [0,1] \to \Omega$ are two different quasihyperbolic geodesics connecting points $x,y\in \Omega$. Then $\gamma_1(0)=\gamma_2(0)=x$, $\gamma_1(1)=\gamma_2(1)=y$, and $\ell_{\mathrm{qh}}(\gamma_1)=\ell_{\mathrm{qh}}(\gamma_2)$. 

We will use the argument in the proof of \cite[4.3]{RasilaTalponen12} but here the 
situation is a bit easier, since the geodesics are assumed to exist.

We may assume that 
\[
\ell_{\mathrm{qh}}(\gamma_1,0,t)=\ell_{\mathrm{qh}}(\gamma_2,0,t)=t\ell_{\mathrm{qh}}(\gamma_1),\textrm{ for all }t\in[0,1].
\]
For $s\in [0,1]$, we define the average path $\gamma_s$ by the formula 
\[\gamma_s(t)=s\gamma_2(t)+(1-s)\gamma_1(t).\] 

Now suppose that $x_1,x_2$, $x_1\neq x_2$ are points such that for some $r_0\in (0,1)$, $\gamma_1(r_0)=x_1$ and $\gamma_2(r_0)=x_2$. As in \cite[4.3]{RasilaTalponen12}, we obtain the estimate
\begin{multline}
\label{r0est}
\ell_{\mathrm{qh}}(\gamma_s,0,r_0) \le s\ell_{\mathrm{qh}}(\gamma_2,0,r_0) + (1-s)\ell_{\mathrm{qh}}(\gamma_1,0,r_0)\\
=\ell_{\mathrm{qh}}(\gamma_1,0,r_0) = \ell_{\mathrm{qh}}(\gamma_2,0,r_0).
\end{multline}

On the other hand, by a similar argument, we have 
\[\ell_{\mathrm{qh}}(\gamma_s,r_0,1)\le \ell_{\mathrm{qh}}(\gamma_1,r_0,1)\le \ell_{\mathrm{qh}}(\gamma_2,r_0,1)\quad 
\mathrm{for\ all}\ s\in [0,1],\] 
and thus equality holds in \eqref{r0est}. It follows that the boundary of the quasihyperbolic ball $D(x,r_0)$ contains a line segment $[x_1,x_2]$. This is a contradiction because in the proof of \cite[Theorem 4.1]{RasilaTalponen12} it was established that the inequality \cite[(4.6)]{RasilaTalponen12} is strict in a strictly convex Banach space with the RNP, and thus, the quasihyperbolic balls are strictly convex (cf. the proof of \cite[Theorem 2.3]{Vaisala09}).
\end{proof}

Similarly as in the argument in \cite[4.3]{RasilaTalponen12} we see that if the distance function $d$ is 
strictly concave (excluding possibly finitely many points of the space), then the QH geodesics are unique.

As mentioned previously, it is known that in any convex domain of a reflexive space there is 
always a geodesic between two points. If the space has the property of being strictly convex (with respect to 
a given norm), then the geodesic is unique. It turns out here, somewhat surprisingly, that if one removes 
the reflexivity assumption, then the statement does not remain valid, even for half spaces.
 
\begin{theorem}\label{thm: non}
Let $(f_{n})\in \ell^{1}\setminus c_{00}$, $\Omega=\{(x_{n})\in c_{0}:\ \sum f_{n}x_{n}>0\}$ and we consider 
$\Omega$ in the quasihyperbolic metric. Given any pair of distinct points $x,y\in \Omega$ there is no geodesic between them.
\end{theorem}

Denote by $\X^*$ the dual of a Banach space $\X$. First, let us recall the following well-known fact:
\begin{fact}
\label{fact1}
Let $f\in \X^*,\ \|f\|=1,$ and $x\in \X$. Then $f(x)=\dist(x,\ker(f))$.
\end{fact}
\begin{proof}
It suffices to check the claim in the case $f(x)=1$. Given $\delta>0$, there is $y\in\X$ 
with $\|y\|<1+\delta$ and $f(y)=1$. Writing $x=y+k$ for some $k\in \ker(f)$ we observe that 
\[\dist(x,\ker(f))=\dist(y,\ker(f))\leq \dist(y,\{0\})=\|y\|< 1+\delta.\]
On the other hand, if $\dist(x,\ker(f))<1$, then there exists $z \in \X,\ \|z\|<1,$ and $h\in \ker(f)$
such that $x=z+h$. This is impossible, since $f(x)=f(z)=1$ and $\|f\|=1$ by the assumptions.
\end{proof}

\begin{proof}[Proof of Theorem \ref{thm: non}]
Without loss of generality we may assume, by rotating the coordinates if necessary, that $f=(f_{n})\in \ell^{1}$ is coordinatewise non-negative. We may also assume that $\|f\|=1$. Observe that $\langle f,x\rangle=\dist_{\|\cdot\|}(x,\partial\Omega)$ for $x\in \Omega$ according to Fact \ref{fact1}. We denote by $M\subset \N$ the infinite subset of indices $m$ such that $f_{m}>0$.

Fix $x=(x_{n}),y=(y_{n})\in \Omega,\ x\neq y$. Assume to the contrary to the statement of the theorem that 
$\gamma\colon [0,1]\to \Omega$ is a quasihyperbolic geodesic joining $x$ and $y$. Let $n\in \N$ be such that $x_{n}-y_{n}\neq 0$. By symmetry we may assume that this difference is positive. We write $x_{n}-y_{n}=\delta> 0$.

We may assume that $\gamma$ is parametrized by its norm length. Clearly $\ell_{\|\cdot\|}(\gamma)\geq \delta$.
 
Next we define a sequence $(\gamma_{k})$ of modifications of $\gamma$ for $k=1,2,\ldots$. Let $e_{n}^{\ast}\colon c_{0}\to \R$ be the functional given by $e_{n}^{\ast}(z)=z_{n}$ for each $z\in c_{0}$. For each $n\in \N$, let 
$\alpha_{n}\colon [0,1]\to \R$ be the polygonal line defined by the following conditions: 
$\alpha_n (0)=\alpha_n (1)=\min(x_n , y_n)$ and $\alpha_n (1/2)=\min(x_n , y_n) + \delta/2$. 
Note that the absolute value of the slope of $\alpha_n$ is $\delta$ (which is defined for values $t\neq 0, 1/2, 1$).
We let $\gamma_{k}$ be such that $e_{n}^{\ast}(\gamma_{k}(t))=\max(e_{n}^* (\gamma(t)),\alpha_n (t))$ for $1\leq n \leq k$, $t\in [0,1]$ and $e_{n}^{\ast}(\gamma_{k}(t))=e_{n}^{\ast}(\gamma(t))$ for $n>k$ and $t\in [0,1]$ (see Figure \ref{fig3}).

\begin{figure}

\includegraphics[width=10cm]{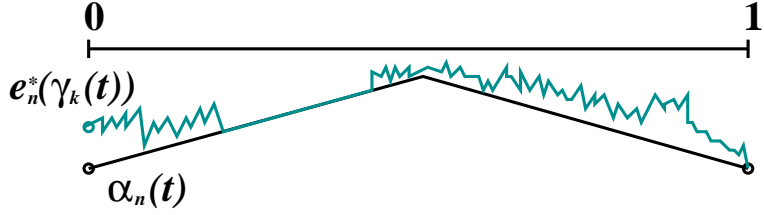}

\caption{Modifications of $\gamma$ projected on a single coordinate.}\label{fig3}
\end{figure}

\noindent{\it Claim:} The inequality $\ell_{\|\cdot\|}(a,b,\gamma_{k})\leq \ell_{\|\cdot\|}(a,b,\gamma)$ holds for 
$0\leq a < b\leq 1$. Indeed, by approximation it suffices to check that the above claim holds in the case where $\gamma$ is a polygonal line with finitely many, say, $p$ line segments. This, in turn, reduces to studying such paths 
supported in only finitely many, say, $m$ first coordinates.

Denote by $c_{00}(m)=[e_i\colon 1\leq i \leq m]\subset c_0$ the corresponding subspace. Since $c_{00}(m)$ is reflexive, being finite-dimensional, it possesses the RNP, and thus we may consider
\[
\gamma(t)=\gamma(0)+\int_{0}^{t} \gamma'(s)\ ds. 
\]
Since the path $\gamma$ is a polygonal line having only finitely many supporting coordinates and line segments,
we may find a finite decomposition $0=t_0< t_1 < \ldots < t_l =1$ of $[0,1]$ in such a way that
$e_{i}^* (\gamma_{k}(t))=\alpha_{i}(t)$ for all $t\in [t_j , t_{j+1}]$ or $e_{i}^* (\gamma_{k}(t))=e_{i}^* (\gamma_{k}(t))$ for all $t\in [t_j , t_{j+1}]$ (or both) for each $j$. Recall that the norm in $c_{00}(m)$ is 
$\|x\|=\max_{1\leq i \leq m}|x_i |$. By using that $|\alpha_{i}'(t)|=\delta \leq \|\gamma'(t)\|$ for a.e. $t$ we obtain the claim.

Note that
\[
\sum_{i\in \N} f_i e_{i}^{\ast}(\gamma_{k}(t))\geq \sum_{i\in \N} f_i e_{i}^{\ast}(\gamma(t)),
\]
for each $k$ and $t$, by the construction of the paths $\gamma_k$. Thus, by applying the above claim we obtain 
$\ell_{\mathrm{qh}}(\gamma_{k})\leq \ell_{\mathrm{qh}}(\gamma)$ for $k$, since $1/\langle f, \gamma\rangle$ is exactly the 
quasihyperbolic weight of a path $\gamma$. Moreover, since $(f_{n})\in \ell^{1}\setminus c_{00}$ and 
$\gamma(1/2)\in c_0$ and $\gamma_k(1/2)\to \delta/2$ as $k\to \infty$ it follows that 
\[\lim_{k\to \infty} \sum_{i\in \N} f_i e_{i}^{\ast}(\gamma_{k}(1/2))>\sum_{i\in \N} f_i e_{i}^{\ast}(\gamma(1/2)).\]
This means that $\ell_{\mathrm{qh}}(\gamma_{k}) < \ell_{\mathrm{qh}}(\gamma)$ for $k$ large enough. This provides us 
with a contradiction, since $\gamma$ was assumed to be a geodesic. 
\end{proof}

Suppose that $\X$ is a Banach space with $\dim \X \ge 2$, and let $\Omega\subsetneq \X$ be an open path-connected domain. Recall that a domain $\Omega$ is geodesic if for every pair of points $x,y\in \Omega$ there exists a geodesic connecting $x$ and $y$. 

\begin{conjecture}
We conjecture that a strictly convex Banach space is reflexive if and only if each of its open half spaces is geodesic in the quasihyperbolic metric. As it was established above, the 'only if' direction holds (since reflexivity 
implies the RNP). 
\end{conjecture}

In the non-strictly convex case the geodesics of half spaces need not be unique, even in finite dimensional 
setting. For example, let $\Omega=\{(x,y)\in \R^2 \colon\ y>0\}\subset c_{00}(2)$ (with the $\sup$-norm). Then the straight line between $(0,1)$ and $(0,2)$ is a quasihyperbolic geodesic. Note that the polygonal line between these two points  passing through $(1/2, 3/2)$ is also a quasihyperbolic geodesic.

Without giving a proof we note that the above half-space example on $c_0$ can be modified in such a way that 
the Banach space can be taken to be even strictly convex. Indeed, a well-known equivalent strictly convex norm 
$|||\cdot |||$ on $c_0$ is given by 
\[|||(x_n)|||^2 = \|(x_n)\|_{c_0}^2 + \sum_{n=1}^{\infty}2^{-n}|x_n|^2\]
and $(f_n)\in \ell^1,\ \|(f_n)\|=1,$ can be selected in such a way that\\ $\limsup_n f_n /2^{-n}=\infty$. 
Note that $|||(f_n)|||^* \leq 1$, so that $\dist(x,\ker f)\geq \langle f,x\rangle$. 

\section{Smoothness of geodesics}

The modulus of convexity $\delta_{\X}$ of Banach space is not necessarily convex. However, it has a largest convex minorant $\delta$. It is easy to see that this is strictly positive if $\delta_{\X}$ is such. In what follows we will use the greatest convex minorant $\delta$ in the place of $\delta_{\X}$. A modulus of convexity $\delta_\X$ is said to have a power type if there is $p\geq 2$ such that $\delta_\X (\epsilon)\geq c \epsilon^p$ for some $c>0$. In such a case we will apply the convex lower bound $c \epsilon^p$ in place of $\delta_\X$. See also Figure \ref{fig1}.

Recall that a function $f$ is H\"older continuous if its modulus of continuity $\nu$ satisfies $\nu(\epsilon)\leq c \epsilon^p$ for some $c,p>0$.

\begin{theorem}\label{thm: }
Let $\X$ be a uniformly convex Banach space whose modulus of convexity has a power type.
Let $d_{w}$ be as above. We assume that $w$ is H\"older continuous on compact sets.
Then every $d_{w}$-geodesic $\gamma$ is $C^1$ excluding the endpoints.
\end{theorem}

It turns out during the course of the proof that it suffices only to consider suitable compact sets 
containing the path. Thus we have the following corollary.

\begin{corollary}
Let $\X$ be a Banach space as above. Then each quasihyperbolic geodesic $\gamma\subset \Omega\subset \X$ is $C^1$-smooth.
\end{corollary}

The argument for the general Banach space here is necessarily different compared to an inner 
product space setting because here there are no angles, etc.

We will first give some auxiliary facts before the proof of the theorem. In what follows $X_1$ and $X_2$ are i.i.d. $\X$-valued random variables with $\|X_1 \|=1$ a.s. In particular, these random variables have (finite) expectations.

\begin{lemma}[Jensen type inequality]\label{lm: Jensen}
Suppose that $X_1$ and $X_2$ are i.i.d. $\X$-valued random variables with $\|X_1 \|=1$ a.s.,
and let $\phi\colon \X\to \R$ be a convex function. Then
\[\E\phi(X_1 -\E X_1)\leq \E\phi(X_1 - X_2).\]
\end{lemma}
\begin{proof}
Since $X_1$ and $X_2$ are independent, we may consider $\P=\P_1 \otimes \P_2$, where $\P_1$ and $\P_2$ 
support the corresponding random variables.
Observe that
\begin{eqnarray*}
\E\phi(X_1 - X_2) &=&\E\phi(X_1 - \E X_1 + \E X_1 - X_2)\\
&=& \int\left(\int \phi(X_1 -\E X_1 + \E X_2 - X_2)d\P_{2}\right)d\P_1\\
&\geq& \int \phi(X_1 -\E X_1) d\P_1= \E\phi(X_1 -\E X_1),
\end{eqnarray*}
where the inequality follows from Jensen's inequality, since $\E(X_2 - \E X_2)=0$.
\end{proof}

\begin{lemma}\label{lm: Ed}
Suppose that $X_1$ and $X_2$ are i.i.d. $\X$-valued random variables with $\|X_1 \|=1$ a.s.
Then
\[\| \E X_1 \| \leq 1-\delta (\E \|X_1 - X_2 \|).\]
\end{lemma}
\begin{proof}
Observe that
\begin{eqnarray*}
\|\E X_1 \| & \leq & \E \frac{\|X_1 + X_2\|}{2} \leq \E(1-\delta(\|X_1 - X_2\|))\\
&=& 1- \E(\delta(\|X_1 - X_2\|))\leq 1- \delta(\E(\|X_1 - X_2\|)).
\end{eqnarray*}
We applied Jensen's inequality and the convexity of $\delta$ in the last estimate.
\end{proof}

\begin{proof}[Proof of Theorem \ref{thm: }]
Let $\gamma\colon [a,b]\to \Omega$ be a (rectifiable) geodesic parameterized by the norm length. 
We aim to show that the restriction $\gamma|_{(a, b)}$ has a continuous derivative.

Since the image of $\gamma$ is compact, its distance to the boundary $\partial \Omega$ is strictly positive, say $d>0$. Let 
\[T\colon \{(s,t,p)\in [a,b]^2 \times [0,1]\colon s\leq t\}\to \X,\] 
\[T(s,t,p)=p\gamma(s)+(1-p)\gamma(t).\] 
Note that this is a uniformly continuous mapping and that the preimage
\[
T^{-1}(\{x\in \X\colon \dist(x,\gamma([a,b]))<d/2) 
\]
is an open neighborhood of the subset $\{(t,t,p)\in [a,b]^2 \times [0,1]\colon \}$. Since $\gamma$ is parameterized 
by its norm-length, it follows that $T$ is uniformly continuous and therefore there is $r_0 >0$ such that 
\begin{multline}
\{(t,\min(t+r,b),p)\colon t\in [a,b],\ p\in [0,1],\ 0\leq r\leq r_0\}\\
\subset  T^{-1}(\{x\in \X\colon \dist(x,\gamma([a,b]))<d/2).
\end{multline}
Let $C$ be the image of the left hand set under the mapping $T$ (see Figure \ref{fig4}). 
Note that $C$ is a compact set as a continuous image of one.

\begin{figure}

\includegraphics[width=6cm]{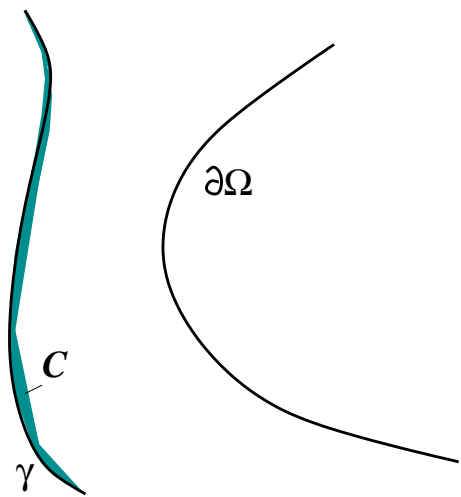}

\caption{The considerations reduce to a compact set of small-width averages of $\gamma$.}\label{fig4}
\end{figure}

Since $\X$ has the RNP, being reflexive, we can recover $\gamma$ by Bochner integrating its 
(vector-valued) derivative $\gamma'$. By the parameterization of $\gamma$ we have that $\|\gamma'\|=1$ a.e.
Define $T\colon [a,b]\times [0,1]\to \X$ by
\[T_{h}(t)=\frac{\int_{[a,b]\cap [t-h,t+h]}\gamma'(s)\ ds}{m([a,b]\cap [t-h,t+h])}\quad \mathrm{for}\ t\in [a,b],\ 0<h<1.\]
Note that 
\[T_{h}(t)=\frac{1}{2h} \int_{[t-h,t+h]}\gamma'\ ds\] 
if $[t-h,t+h]\subset [a,b]$ which is the essential case here.
Also note that 
\[T(t-h,t+h,1/2)-\gamma(t-h)=\frac{1}{2}\int_{t-h}^{t+h}\gamma'\ ds=hT_{h}(t).\]
By the Lebesgue differentiation theorem $\lim_{h\to 0^{+}}T_{h}(t)=\gamma'(t)$ for a.e. $t$. Thus, without loss
of generality may assume by redefining $\gamma'$ that these coincide everywhere and we write $T_{0}\equiv\gamma'$.

By a standard argument involving approximation of $\gamma'$ by simple functions we observe that for a fixed $h>0$ 
the map $[a,b]\to \X,\ t\mapsto T_{h}(t)$ is continuous. 

Therefore it suffices to check that 
\[\lim_{h\to 0^+}\sup_{t\in [a_1,b_1]}\|T_{h}(t)-\gamma'(t)\|=0\] 
for any closed interval $[a_1,b_1]\subset (a,b)$ because then $\gamma'$ will be continuous 
on $[a_1,b_1]$ and the statement of the theorem follows. To this end, we are actually required to verify that 
$T_{h/2^n}$, considered as mappings $[a_1,b_1]\to \X$, form a Cauchy sequence in $C([a_1,b_1],\X)$. 
It suffices to establish an estimate
\begin{equation}\label{eq: (ast)}   
\frac{1}{2h}\int_{t-h}^{t+h}\|T_{h}(t)-\gamma'(s)\|\ ds \leq \beta(h)
\end{equation}
where $\beta(h)$ tends to $0$ suitably rapidly as $h\to 0$ (not depending on $t\in [a_1,b_1]$). 
Indeed, we will check that $T_{h/2^n}$ is $C([a_1,b_1],\X)$-Cauchy. Note that 
\begin{multline*}
\|T_{h}(t)-T_{h/2}(t)\|=\|\frac{1}{2h}\int_{t-h}^{t+h}\gamma'(s)\ ds - \frac{1}{h}\int_{t-h/2}^{t+h/2}\gamma'(s)\ ds\|\\
 =\|\frac{1}{2h}\int_{t-h\leq s \leq 1-h/2\vee t+h/2\leq s\leq t+h}\gamma'(s)\ ds - \frac{1}{2h}\int_{t-h/2}^{t+h/2}\gamma'(s)\ ds\|\\
 =\frac{1}{2h}\|\int_{t-h\leq s \leq 1-h/2\vee t+h/2\leq s\leq t+h}\gamma'(s) -T_h(t)\ ds - \int_{t-h/2}^{t+h/2}\gamma'(s)-T_h(t)\ ds\|\\
\leq  \frac{1}{2h}\int_{t-h}^{t+h}\|\gamma'(s)-T_h (t)\|\ ds\leq \beta(h).
\end{multline*}  
Now $(T_{h/2^n})$ is Cauchy, since 
\begin{equation}\label{eq: suminfty}
\|T_{h}-T_{h/2^n}\|\leq \sum_{i=1}^{n} \|T_{h/2^{i-1}}-T_{h/2^i}\|\leq \sum_{i=0}^{\infty} \beta(h/2^i)<\infty
\end{equation}
where the convergence of the right hand sum will be established subsequently.

\begin{figure}

\includegraphics[width=8cm]{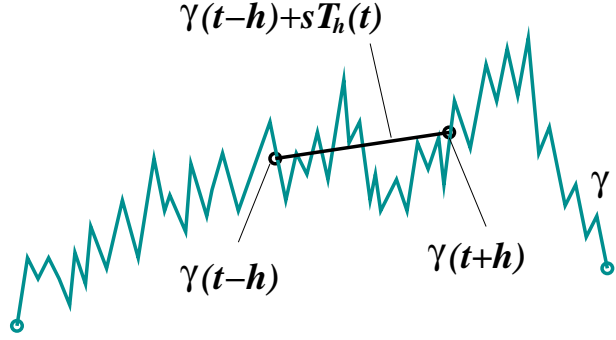}

\caption{An approximation of a path $\gamma$ by using the moving average $T_h$ of $\gamma'$.}\label{fig5}
\end{figure}

Put $h_{0}=r_0 /2$ and observe that (see Figure \ref{fig5})
\[
\{\gamma(t-h)+s T_{h}(t)\colon 0\leq s\leq 2h\}\subset C
\]
for $0\leq h\leq h_{0}$, $[t-h,t+h]\subset [a,b]$. Note that $w|_{C}$ is H\"older continuous and let $\nu(\epsilon)$ be the modulus of continuity of $w|_{C}$. Thus $\nu(\epsilon)\leq c_1 \epsilon^r$ for some $r>0$ and $\nu(\epsilon) \to 0$ as $\epsilon\to 0$.

Since $\gamma$ is a $d_{w}$-geodesic we have that 
\begin{equation}
\begin{split}
(w(\gamma(t))-\nu(h))\int_{t-h}^{t+h}\|d\gamma\|&\leq (w(\gamma(t-h))+\nu(2h))2h\|T_{h}(t)\|\\
&\leq  (w(\gamma(t))+\nu(3h))2h\|T_h (t)\|.\\
\end{split}
\end{equation}
The above estimate uses the facts that $\|\gamma(t)-\gamma(s)\|\leq h$ for $|t-s|\leq h$ 
and $|\gamma(t-h)-(\gamma(t-h)+sT_{h}(t))|\leq 2h$ for $0\leq s \leq 2h$. Note that $s\mapsto \gamma(t-h)+sT_{h}(t)$ 
defines a straight line $[\gamma(t-h),\gamma(t+h)]$ and $2h\|T_{h}(t)\|$ is the length of the line. 
It follows that 
\begin{equation}
\begin{split}
\frac{1}{2h}\int_{t-h}^{t+h}\|d\gamma\| -\|T_h (t)\|&\leq  \left(\frac{w(\gamma(t))+\nu(3h)}{w(\gamma(t))-\nu(h)}-1\right)\\
&\leq  \frac{2\nu(3h)}{w_0}\\
\end{split}
\end{equation}
for a suitable $w_0 >0$ and sufficiently small $h>0$. Note that since $C$ is compact the weight $w$ attains 
its minimum on $C$ which must be greater than $0$. In what follows we will consider only $h$ such that
the last inequality above holds. For convenience we will abbreviate the last term of the above inequality by $\mu(h)=2\nu(3h)/w_0$.

Note that
\begin{multline}\label{eq: searrow}
\frac{1}{2h}\int_{t-h}^{t+h}\|d\gamma\|-\|T_{h}(t)\|\\ 
=\frac{1}{2h}\int_{t-h}^{t+h}\|d\gamma\| -\frac{1}{2h}\|\gamma(t+h)-\gamma(t-h)\| \leq \mu(h)
\end{multline}
where $\mu(h)$ does not depend on $t$ and $\mu(h) \searrow 0$ as $h\to 0$.
 
Therefore on $[t-h,t+h]$ the path $\gamma$ has '$\mu$-asymptotically' the same length as its linear approximation $s\mapsto \gamma(t-h)+sT_{h}(t)$, $s\in [0,2h]$. This is controlled by $\nu(h)$.
The philosophy is that if the space is uniformly convex then $\gamma$ is heavily penalized for squiggling around its linear approximation. We will justify \eqref{eq: (ast)} by using \eqref{eq: searrow} and the modulus of uniform convexity.

To finish the argument, let us consider the state space $[t-h,t+h]^2$ with the probability measure 
\[m(A)=\frac{1}{2h}\int_{t-h}^{t+h}\frac{1}{2h}\int_{t-h}^{t+h}1_{A}(s,r)\ dr\ ds.\]
Put $X_1(s,r)=\gamma'(s)$ and $X_2(s,r)=\gamma'(r)$ and of course we implicitly assume these random variables
depend on $t$ and $h$, although the estimates depend on $h$ only. According to \eqref{eq: searrow} we obtain that 
$\|\E X_1\|\geq 1-\mu(h)$ so that $\delta_{\X}(\E\|X_1 - X_2\|)\leq \mu(h)$ according to Lemma \ref{lm: Ed}. 
Here we used the lemma with $\delta (\epsilon) = c \epsilon^p$, which is convex since we may take $p\geq 2$. 
Thus $c(\E\|X_1 - X_2\|)^p \leq \mu(h)$ so that by applying Lemma \ref{lm: Jensen} we get
\[\E\|X_1 - \E X_1\| \leq \E\|X_1 - X_2\| \leq \frac{1}{c}(\mu(h))^{\frac{1}{p}} \leq c_2 h^{\frac{r}{p}} = \beta(h),\]
the last equality being a definition. Observe that 
\[\sum_{i=1}^{\infty}\beta(h/2^i)=\sum_{i=1}^{\infty}c_2 (h/2^i)^{\frac{r}{p}}\] 
converges as a geometric series with ratio $1/2^{\frac{r}{p}}$. We note that 
\[\E\|X_1 - \E X_1\|=\frac{1}{2h}\int_{t-h}^{t+h} \|\gamma'(s) - T_h (t)\|\ ds,\]
so that we have verified \eqref{eq: (ast)} and \eqref{eq: suminfty}.
\end{proof}

It was recently brought to our attention that J. V\"ais\"al\"a has independently arrived at a result similar 
the one above by a different argument \cite{Vaisala12}.

We call a curve $\gamma\colon \R \to \Omega$ a \emph{local geodesic} if for each $a,b\in\R,\ a<b,$ the restriction $\gamma|_{[a,b]}$ is a geodesic between $\gamma(a)$ and $\gamma(b)$. The above result extends naturally to the local geodesics as well.

\subsection*{Acknowledgments} 
We thank J.~V\"ais\"al\"a and M.~Vuorinen for their helpful comments on this paper.

\end{document}